\documentclass[10pt]{amsart}
\usepackage{amssymb,mathrsfs}
\usepackage{lmodern}
\usepackage{color}

\def\bR {\mathbf{R}}
\def\bS {\mathbf{S}}

\def\bZ {\mathbf{Z}}

\def\cB {\mathcal{B}}
\def\cC {\mathcal{C}}
\def\cD {\mathcal{D}}

\def\cL {\mathcal{L}}

\def\cN {\mathcal{N}}

\def\cQ {\mathcal{Q}}
\def\cR {\mathcal{R}}

\def\cT {\mathcal{T}}

\def\a {{\alpha}}
\def\b {{\beta}}
\def\g {{\gamma}}

\def\de {{\delta}}
\def\eps {{\epsilon}}
\def\th {{\theta}}

\def\ka {{\kappa}}

\def\si {{\sigma}}
\def\Si {{\Sigma}}

\def\om {{\omega}}

\def\d {{\partial}}
\def\grad {{\nabla}}
\def\Dlt {{\Delta}}

\def\la {\langle}
\def\ra {\rangle}
\def \La {\bigg\langle}
\def \Ra {\bigg\rangle}

\def\wto {{\rightharpoonup}}

\def\dd {\mathrm{d}}

\newcommand{\Div}{\operatorname{div}}

\newcommand{\Dom}{\operatorname{Dom}}

\newcommand{\Span}{\operatorname{span}}

\newcommand{\Ker}{\operatorname{Ker}}
\newcommand{\Img}{\operatorname{Im}}

\newcommand{\ba}{\begin{aligned}}
\newcommand{\ea}{\end{aligned}}

\newcommand{\be}{\begin{equation}}
\newcommand{\ee}{\end{equation}}
\newcommand{\lb}{\label}

\newtheorem{theorem}{Theorem}[section]

\newtheorem{proposition}{Proposition}

\theoremstyle{definition}

\newtheorem{remark}{Remark}

\definecolor{pervinca}{rgb}{0.4, 0.3, 0.8}
\definecolor{violetto}{rgb}{0.8, 0.4, 0.8 }

\begin{document}

\title[Vlasov-Stokes for Aerosol Flows]{A Derivation of the Vlasov-Stokes System\\ for Aerosol Flows from the KInetic Theory\\ of Binary Gas Mixtures}

\author[E. Bernard]{Etienne Bernard}
\address[E.B.]{IGN-LAREG, Universit\'e Paris Diderot, B\^atiment Lamarck A, 5 rue Thomas Mann, Case courrier 7071, 75205 Paris Cedex 13, France}
\email{esteve.bernard@gmail.com}

\author[L. Desvillettes]{Laurent Desvillettes}
\address[L.D.]{Universit\'e Paris Diderot, Sorbonne Paris Cit\'e, Institut de Math\'ematiques de Jussieu - Paris Rive Gauche, UMR CNRS 7586, CNRS, Sorbonne Universit\'es, UPMC Univ. Paris 06, 75013, Paris, France}
\email{desvillettes@math.univ-paris-diderot.fr}

\author[F. Golse]{Fran\c cois Golse}
\address[F.G.]{CMLS, Ecole polytechnique et CNRS, Universit\'e Paris-Saclay, 91128 Palaiseau Cedex, France}
\email{francois.golse@polytechnique.edu}

\author[V. Ricci]{Valeria Ricci}
\address[V.R.]{Dipartimento di Matematica e Informatica, Universit\`a degli Studi di Palermo,
Via Archirafi 34, I90123 Palermo, Italy}
\email{valeria.ricci@unipa.it}

\begin{abstract}
In this short paper, we formally derive the thin spray equation for a steady Stokes gas (i.e. the equation consists in a  coupling between a kinetic --- Vlasov type --- equation for the dispersed phase and a --- steady --- Stokes equation 
for the gas). Our starting point is a system of Boltzmann equations for a binary gas mixture. The derivation follows the procedure already outlined in [Bernard-Desvillettes-Golse-Ricci, {\tt arXiv:1608.00422}[math.AP]] where 
the evolution of the gas is governed by the Navier-Stokes equation.
\end{abstract}



\maketitle


\section{Introduction}


An aerosol or a spray is a fluid consisting of a \textit{dispersed phase}, usually liquid droplets, sometimes solid particles, immersed in a gas referred to as the \textit{propellant}.

An important class of models for the dynamics of aerosol/spray flows consists of

(a) a kinetic equation for the dispersed phase, and

(b) a fluid equation for the propellant.

The kinetic equation for the dispersed phase and the fluid equation for the propellant are coupled through the drag force exerted by the gas on the droplets/particles. This class of models applies to the case of \textit{thin} sprays, i.e. 
those for which the volume fraction of the dispersed phase is typically $\ll 1$.

Perhaps the simplest example of this class of models is the Vlasov-Stokes system:
$$
\left\{
\ba
{}&\d_tF+v\cdot\grad_xF-\frac{\ka}{m_p}\Div_v((v-u)F)=0\,,
\\
&-\rho_g\nu\Dlt_xu=-\grad_xp+\ka\int_{\bR^3}(v-u)F\,\dd v\,,
\\
&\Div_xu=0\,.
\ea
\right.
$$
The unknowns in this system are $F\equiv F(t,x,v)\ge 0$, the distribution function of the dispersed phase, i.e. the number density of particles or droplets with velocity $v$ located at the position $x$ at time $t$, and $u\equiv u(t,x)\in\bR^3$,
the velocity field in the gas. The parameters $\ka$, $m_p$, $\rho_g$ and $\nu$ are positive constants. Specifically, $\ka$ is the friction coefficient of the gas on the dispersed phase, $m_p$ is the mass of a particle or droplet, and $\rho_g$ 
is the gas density, while $\nu$ is the kinematic viscosity of the gas. The aerosol considered here is assumed for simplicity to be \textit{monodisperse} --- i.e. all the particle in the dispersed phase are of the same size and of the same mass. 
In practice, the particles in the dispersed phase of an aerosol are in general distributed in size (and in mass).

The last equation in the system above indicates that the gas flow is considered as incompressible\footnote{It is well known that the motion of a gas at a very low Mach number is governed by the equations of incompressible fluid mechanics, 
even though a gas is a compressible fluid. A formal justification for this fact can be found on pp. 11--12 in \cite{PLLFluidMech1}.}. The scalar pressure field $p\equiv p(t,x)\in\bR$ is instantaneously coupled to the unknowns $F$ and $u$ by 
the Poisson equation
$$
\Dlt_xp=\ka\Div_x\int_{\bR^3}(v-u)F\,\dd v\,.
$$
The mathematical theory of the Vlasov-Stokes system has been discussed in \cite{Hamdache} --- see in particular section 6 there, which treats the case of a steady Stokes equation as above.

Our purpose in the present work is to provide a rigorous derivation of this system from a more microscopic system.

Derivations of the Stokes equation with a force term including the drag force exerted by the particles on the fluid (known as the Brinkman force) from a system consisting of a large number of particles immersed in a viscous fluid can be 
found in \cite{Allaire,DesvFGRicci08}. Both results are based on the method of homogenization of elliptic operators on domains with holes of finite capacity, pioneered by Khruslov and his school --- see for instance \cite{Khruslov,CioraMurat}.
Unfortunately, this method assumes that the minimal distance between particles remains uniformly much larger than the particle radius $r\ll 1$. Specifically, this minimal distance is assumed in \cite{DesvFGRicci08} to be of the order of
$r^{1/3}$ in space dimension $3$; this condition has been recently improved in \cite{Hillairet}. Such particle configurations are of vanishing probability as the particle number $N\to\infty$: see for instance Proposition 4 in \cite{Hauray}. 
Moreover, the question of propagating this separation condition by the particle dynamics seems open so far --- see however \cite{JabinOtto} for interesting ideas on a similar problem for a first order dynamics.

For that reason, we have laid out in \cite{BDGR} a program for deriving dynamical equations for aerosol flows from a system of Boltzmann equations for the dispersed phase and the propellant viewed as a binary gas mixture. In \cite{BDGR},
we have given a complete formal derivation of the Vlasov-(incompressible) Navier-Stokes system from a scaled system of Boltzmann equations. We have identified the scaling leading to this system, which involves two small parameters.
One is the mass ratio $\eta$ of the gas molecule to the particle in the dispersed phase. The other small parameter is the ratio $\eps$ of the thermal speed of the dispersed phase to the speed of sound in the gas. The assumption $\eta\ll 1$ 
implies that the gas molecules impingement on the particles in the dispersed phase results in a slight deviation of these particles, and this accounts for the replacement of one of the collision integrals in the Boltzmann system by a Vlasov 
type term. The assumption $\eps\ll 1$ explains why a low Mach number approximation is adequate for the motion equation in the propellant. In particular, the velocity field in the propellant is approximately divergence free, and the motion 
equation in the gas is the same as in an incompressible fluid with constant density.

However, a more intricate scaling is needed to derive the Vlasov-Stokes system above from the sytem of Boltzmann equations for a binary mixture. If the ratio $\mu$ of the mass density of the propellant to the mass density of the dispersed 
phase is very small, and the thermal speed in the dispersed phase is much smaller than that of the propellant, one can hope that the friction force exerted by the dispersed phase on the propellant will slow down the gas, so that Navier-Stokes 
motion equation can be replaced with a Stokes equation. Although this scenario sounds highly plausible, the asymptotic limit of the system of Boltzmann equations for a binary gas mixture leading to the Vlasov-Stokes system rests on a
rather delicate tuning of the three small parameters $\eps,\eta,\mu$, defined in the statement of our main result, Theorem \ref{theor}.

The outline of this paper is as follows: section \ref{S-S2} introduces the system of Boltzmann equation for binary gas mixtures, identifies the scaling parameters involved in the problem, and presents two classes of Boltzmann type collision 
integrals describing the interaction between the dispersed phase and the propellant. Section \ref{S-S3} formulates a few (specifically, five) key abstract assumptions on the interaction between the dispersed phase and the propellant, which 
are verified by the models introduced in section \ref{S-S2}. The main result of the present paper, i.e. the derivation of the Vlasov-Stokes system from the system of Boltzmann equations for a binary gas mixture, is Theorem \ref{theor}, stated 
at the begining of section \ref{S-S4}. The remaining part of section \ref{S-S4} is devoted to the proof of Theorem \ref{theor}.

Obviously, the present paper shares many features with its companion \cite{BDGR} --- we have systematically used the same notation in both papers. However, the derivation of the Vlasov-Stokes system differs in places from that of the
Vlasov-Navier-Stokes system in \cite{BDGR}. For instance, some assumptions on the interaction between the propellant and the dispersed phase used in the present paper are slightly different from their analogues in \cite{BDGR}. 
We have therefore kept the repetitions between \cite{BDGR} and the present paper to a strict minimum. Only the part of the proof of Theorem \ref{theor} that is special to the derivation of the Vlasov-Stokes system is given in full detail. The 
reader is referred to \cite{BDGR} for all the arguments which have been already used in the derivation of the Vlasov-Navier-Stokes system.


\section{Boltzmann Equations for Multicomponent Gases}\lb{S-S2}


Consider a binary mixture consisting of microscopic gas molecules and much bigger solid dust particles or liquid droplets. For simplicity, we henceforth assume that the dust particles or droplets are identical (in particular, the spray is 
monodisperse: all particles have the same mass), and that the gas is monatomic. We denote from now on by $F\equiv F(t,x,v)\ge 0$ the distribution function of dust particles or droplets, and by $f\equiv f(t,x,w)\ge 0$ the distribution function 
of gas molecules. These distribution functions satisfy the system of Boltzmann equations
\be\lb{BoltzSys0}
\ba
(\d_t+v\cdot\grad_x)F&=\cD(F,f)+\cB(F)\,,
\\
(\d_t+w\cdot\grad_x)f&=\cR(f,F)+\cC(f)\,.
\ea
\ee
The terms $\cB(F)$ and $\cC(f)$ are the Boltzmann collision integrals for pairs of dust particles or liquid droplets and for gas molecules respectively. The terms $\cD(F,f)$ and $\cR(f,F)$ are Boltzmann type collision integrals describing 
the deflection of dust particles or liquid droplets subject to the impingement of gas molecules, and the slowing down of gas molecules by collisions with dust particles or liquid droplets respectively.

Collisions between molecules are assumed to be elastic, and satisfy therefore the usual local conservation laws of mass, momentum and energy, while collisions between dust particles may not be perfectly elastic, so that $\cB(F)$ satisfies 
only the local conservation of mass and momentum. Since collisions between gas molecules and particles preserve the nature of the colliding objects, the collision integrals $\cD$ and $\cR$ satisfiy the local conservation laws of particle 
number per species and local balance of momentum. The local balance of energy is satisfied if the collisions between gas molecules and particles are elastic.

The system (\ref{BoltzSys0}) is the starting point in our derivation of the Vlasov-Navier-Stokes system in \cite{BDGR}. We shall mostly follow the derivation in \cite{BDGR}, and shall insist only on the differences between the limit considered
there and the derivation of the Vlasov-Stokes system studied in the present paper.

\subsection{Dimensionless Boltzmann systems}


We assume for simplicity that the aerosol is enclosed in a periodic box of size $L>0$, i.e. $x\in\bR^3/L\bZ^3$. The system of Boltzmann equations (\ref{BoltzSys0}) involves an important number of physical parameters, which are listed in 
the table below.


\bigskip
\begin{center}
\begin{tabular}{|c|c|}
\hline
\hspace{.2cm} Parameter \hspace{.2cm} & \hspace{.2cm} Definition \hspace{.2cm}\\
\hline
\hline
\hspace{.2cm} $L$ \hspace{.2cm} & \hspace{.2cm} size of the container (periodic box) \hspace{.2cm}\\
\hline
\hspace{.2cm} $\cN_p$ \hspace{.2cm} & \hspace{.2cm} number of particles$/L^3$ \hspace{.2cm}\\
\hline
\hspace{.2cm} $\cN_g$ \hspace{.2cm} & \hspace{.2cm} number of gas molecules$/L^3$ \hspace{.2cm}\\
\hline
\hspace{.2cm} $V_p$ \hspace{.2cm} & \hspace{.2cm} thermal speed of particles \hspace{.2cm}\\
\hline
\hspace{.2cm} $V_g$ \hspace{.2cm} & \hspace{.2cm} thermal speed of gas molecules \hspace{.2cm}\\
\hline
\hspace{.2cm} $S_{pp}$ \hspace{.2cm} & \hspace{.2cm} average particle/particle cross-section \hspace{.2cm}\\
\hline
\hspace{.2cm} $S_{pg}$ \hspace{.2cm} & \hspace{.2cm} average particle/gas cross-section \hspace{.2cm}\\
\hline
\hspace{.2cm} $S_{gg}$ \hspace{.2cm} & \hspace{.2cm} average molecular cross-section \hspace{.2cm}\\
\hline
\hspace{.2cm} $\eta=m_g/m_p$ \hspace{.2cm} & \hspace{.2cm} mass ratio (molecules/particles) \hspace{.2cm}\\
\hline
\hspace{.2cm} $\mu=(m_g \cN_g)/(m_p \cN_p)$ \hspace{.2cm} & \hspace{.2cm} mass fraction (gas/dust or droplets) \hspace{.2cm}\\
\hline
\hspace{.2cm} $\eps=V_p/V_g$ \hspace{.2cm} & \hspace{.2cm} thermal speed ratio (particles/molecules) \hspace{.2cm}\\
\hline
\end{tabular}
\end{center}


\smallskip
This table of parameters is the same as in \cite{BDGR}, except for the mass fraction $\mu$ which does not appear in \cite{BDGR}.

\bigskip
We first define a dimensionless position variable:
$$
\hat x:=x/L\,,
$$
together with dimensionless velocity variables for each species:
$$
\hat v:=v/V_p\,,\quad \hat w:=w/V_g\,.
$$
In other words, the velocity of each species is measured in terms of the thermal speed of the particles in the species under consideration. 

Next, we define a time variable, which is adapted to the slowest species, i.e. the dust particles or droplets:
$$
\hat t:=tV_p/L\,.
$$

Finally, we define dimensionless distribution functions for each particle species:
$$
\hat F(\hat t,\hat x,\hat v):=V^3_pF(t,x,v)/\cN_p\,,\qquad\hat f(\hat t,\hat x,\hat w):=V^3_gf(t,x,w)/\cN_g\,.
$$

The definition of dimensionless collision integrals is more complex and involves the average collision cross sections $S_{pp},S_{pg},S_{gg}$, whose definition is recalled below. 

The collision integrals $\cB(F)$, $\cC(f)$, $\cD(F,f)$ and $\cR(f,F)$ are given by expressions of the form
\be
\label{Colli}
\ba
\cB(F)(v)=&\iint_{\bR^3\times\bR^3}F(v')F(v'_*)\Pi_{pp}(v,\dd v'\,\dd v'_*)
\\
&-F(v)\int_{\bR^3}F(v_*)|v-v_*|\Si_{pp}(|v-v_*|)\,\dd v_*\,,
\\
\cC(f)(w)=&\iint_{\bR^3\times\bR^3}f(w')f(w'_*)\Pi_{gg}(w,\dd w'\,\dd w'_*)
\\
&-f(w)\int_{\bR^3}f(w_*)|w-w_*|\Si_{gg}(|w-w_*|)\,\dd w_*\,,
\\
\cD(F,f)(v)=&\iint_{\bR^3\times\bR^3}F(v')f(w')\Pi_{pg}(v,\dd v'\,\dd w')
\\
&-F(v)\int_{\bR^3}f(w)|v-w|\Si_{pg}(|v-w|)\,\dd w\,,
\\
\cR(f,F)(w)=&\iint_{\bR^3\times\bR^3}F(v')f(w')\Pi_{gp}(w,\dd v'\,\dd w')
\\
&-f(w)\int_{\bR^3}F(v)|v-w|\Si_{pg}(|v-w|)\,\dd v\,.
\ea
\ee
In these expressions, $\Pi_{pp},\Pi_{gg},\Pi_{pg},\Pi_{gp}$ are nonnegative, measure-valued measurable functions defined a.e. on $\bR^3$, while $\Si_{pp},\Si_{gg},\Si_{pg}$ are nonnegative measurable functions defined 
a.e. on $\bR_+$. This setting is the same as in \cite{BDGR}, and is taken from chapter 1 in \cite{Landau10} (see in particular formula (3.6) there).

The relation between the quantities $\Pi$ and $\Si$ is the following:
\begin{equation}\label{Colli2}
\ba 
\int_{\bR^3}\Pi_{pp}(v,\dd v'\,\dd v'_*) \, \dd v =  |v'-v'_*|\Si_{pp}(|v'-v'_*|)\dd v'\,\dd v'_*,
\\
\int_{\bR^3} \Pi_{gg}(w,\dd w'\,\dd w'_*) \, \dd w = |w'-w'_*|\Si_{gg}(|w'-w'_*|)\dd w'\,\dd w'_*,
\\
\int_{\bR^3} \Pi_{pg}(v,\dd v'\,\dd w') \, \dd v = |v'-w'|\Si_{pg}(|v'-w'|)\dd v'\,\dd w',
\\
\int_{\bR^3} \Pi_{gp}(w,\dd v'\,\dd w') \, \dd w = |v'-w'|\Si_{pg}(|v'-w'|)\dd v'\,\dd w'.
\ea
\end{equation}

The dimensionless quantities associated to $\Si_{pp},\Si_{gg}$ and $\Si_{pg}$ are ($i,j=p,g$)
$$
\ba
\hat\Si_{ii}(|\hat z|)&=\Si_{ii}(V_i|\hat z|)/S_{ii}\,,
\\
\hat\Si_{ij}(|\hat z|)&=\Si_{ij}(V_j|\hat z|)/S_{ij}\,.
\ea
$$
Likewise
$$
\ba
\hat\Pi_{pp}(\hat v,\dd\hat v'\,\dd\hat v'_*)&=\Pi_{pp}(v,\dd v'\,\dd v'_*)/S_{pp}V_p^4\,,
\\
\hat\Pi_{gg}(\hat w,\dd\hat w'\,\dd\hat w'_*)&=\Pi_{gg}(w,\dd w'\,\dd w'_*)/S_{gg}V_g^4\,,
\\
\hat\Pi_{pg}(\hat v,\dd\hat v'\,\dd\hat w')&=\Pi_{pg}(v,\dd v'\,\dd w')/S_{pg}V_g^4\,,
\\
\hat\Pi_{gp}(\hat w,\dd\hat v'\,\dd\hat w')&=\Pi_{gp}(w,\dd v'\,\dd w')/S_{pg}V_gV_p^3\,.
\ea
$$

With the dimensionless quantities so defined, we arrive at the following dimensionless form of the multicomponent Boltzmann system:
\be\lb{BoltzSys}
\left\{
\ba
{}&\d_{\hat t}\hat F\,+\,\hat v\cdot\grad_{\hat x}\hat F\,=\cN_gS_{pg}L\frac{V_g}{V_p}\hat\cD(\hat F,\hat f)+\cN_pS_{pp}L\hat\cB(\hat F)\,,
\\
&\d_{\hat t}\hat f\!+\!\frac{V_g}{V_p}\hat w\!\cdot\!\grad_{\hat x}\hat f=\cN_pS_{pg}L\frac{V_g}{V_p}\hat\cR(\hat f,\hat F)+\cN_gS_{gg}L\frac{V_g}{V_p}\hat\cC(\hat f)\,.
\ea
\right.
\ee

Throughout the present study, we shall always assume that 
\be\lb{NoppColl}
\cN_pS_{pp}L\ll 1\,,
\ee
so that the collision integral for dust particles or droplets $\cN_pS_{pp}L\hat\cB(\hat F)$   is considered as formally negligible (and will not appear anymore in the equations).

Besides, the thermal speed $V_p$ of dust particles or droplets is in general smaller than the thermal speed $V_g$ of gas molecules; thus we denote their ratio by
\be\lb{Def-eps}
\eps=\frac{V_p}{V_g}\in[0,1]\,.
\ee

Recalling that the mass ratio $[0,1]\ni \eta = m_g/m_p$ is supposed to be extremely small, since the particles are usually much heavier than the molecules, we also assume 
\be\lb{Def-eta}
\frac{\eta}{\mu}=\frac{\cN_p}{\cN_g}\in[0,1]\,.
\ee
where $\mu$ is the mass fraction of the gas with respect to the droplets, which is also supposed to be extremely small. This hypothesis on  the mass ratio gives a scaling such that the mass density of the gas is very small with respect 
to the mass density of the dispersed phase.

Finally, in the sequel (cf. eq. (\ref{kifu})), we shall assume that 
$$
\cN_p\,S_{pg}\, L =\frac{\eps}{\mu}\quad\hbox{ and }\quad\cN_g\,S_{gg}\, L = \frac{\mu}{\eps}\,,\qquad\quad\hbox{ where }\eps\ll\mu\ll 1\,.
$$ 
With these assumptions, one has
$$
\ba
(\cN_gS_{pg}L)\frac{V_g}{V_p}=(\frac{\cN_g}{\cN_p})(\cN_pS_{pg}L)(\frac{V_g}{V_p})=\frac1\eta\,,
\\	\\
(\cN_pS_{pg}L)(\frac{V_g}{V_p})=\frac1\mu\,,
\quad
(\cN_gS_{gg}L)(\frac{V_g}{V_p})=\frac{\mu}{\eps^2}\,,
\\	\\
\cN_pS_{pp}L\ll 1\,,
\ea
$$
so that we arrive at the scaled system:
\be\lb{BoltzSysSc}
\left\{
\ba
{}&\d_{\hat t}\hat F\,+\,\hat v\cdot\grad_{\hat x}\hat F\,=\frac{1}{\eta}\hat\cD(\hat F,\hat f)\,,
\\
&\d_{\hat t}\hat f+\frac{1}{\eps}\hat w\cdot\grad_{\hat x}\hat f=\frac1\mu\hat\cR(\hat f,\hat F)+\frac{\mu}{\eps^2}\hat\cC(\hat f)\,.
\ea
\right.
\ee

Henceforth, we drop hats on all dimensionless quantities and variables introduced in this section. Only dimensionless variables, distribution functions and collision integrals will be considered from now on. 

We also use $V,W$ as variables in the positive part of the collision operators $\cD$ and $\cR$, in order to avoid confusions.

We define therefore the ($\eps$- and $\eta$-dependent) dimensionless collision integrals
\begin{equation} \label{newc}
 \ba 
 \cC(f)( w)=&\iint_{\bR^3\times\bR^3}f(w') f(w'_*) \Pi_{gg}(w,\dd w'\,\dd w'_*)
\\ 
&- f(w)\int_{\bR^3} f(w_*)| w- w_*|  \Si_{gg}(|w- w_*|)\,\dd w_*\,,
\ea 
\end{equation} 
\begin{equation} \label{newd}
\ba 
\cD( F, f)( v)=&\iint_{\bR^3\times\bR^3} F(V)f(W)\Pi_{pg}( v,\dd V\,\dd W)
\\
&- F(v)\int_{\bR^3} f(w)\left|\eps v- w\right|\Si_{pg}\left(\left|\eps v- w\right|\right)\,\dd w\,,
\ea
\end{equation} 
\begin{equation} \label{newr}
\ba 
\cR( f, F)( w)=&\iint_{\bR^3\times\bR^3} F(V) f(W) \Pi_{gp}( w,\dd V\,\dd W)
\\
&- f( w)\int_{\bR^3} F(v)\left|\eps v- w\right| \Si_{pg}\left(\left|\eps v-w\right|\right)\,\dd  v\,,
\ea
\end{equation} 
with $\Si_{gg}$, $\Si_{pg}$ satisfying (\ref{Colli2}). Notice that $\Pi_{pg}$ and $\Pi_{gp}$ depend in fact on $\eps$ and $\eta$, and will sometimes be denoted by $\Pi_{pg}^{\eps,\eta}$ and $\Pi_{gp}^{\eps,\eta}$, whenever needed.

With the notation defined above, the scaled Boltzmann system (\ref{BoltzSysSc}) is then recast as:
\be\lb{BoltzSysSc2}
\left\{
\ba
{}&\d_t F\,+\,v\cdot\grad_x F\,=\frac{1}{\eta}\cD(F,f)\,,
\\
&\d_t f+\frac{1}{\eps}w\cdot\grad_x f=\frac1\mu\cR(f, F)+\frac{\mu}{\eps^2}\cC(f)\,.
\ea
\right.
\ee

At this point, it may be worthwhile explaining the difference between the scalings considered in the present paper and in \cite{BDGR}. 

In \cite{BDGR}, we implicitly assumed that $\mu=1$, while we have assumed in the present paper that $\mu\ll 1$. 

When $\mu\ll 1$, the density of the dispersed phase is much higher than that of the propellant, and since $\eps\ll 1$, the thermal speed of the dispersed phase is much smaller than that of the gas. Therefore the dispersed phase
slows down the motion of gas molecules, so that the Reynolds number in the gas becomes small and the material derivative in the Navier-Stokes equation becomes negligible. This qualitative argument explains why the scaling
considered in the present paper leads to a steady Stokes equation in the gas, while the assumption $\mu=1$ as in \cite{BDGR} leads to a Navier-Stokes equation.

\subsection{The explicit form of collision integrals in two physical situations}


\subsubsection{The Boltzmann collision integral for gas molecules}


The dimensionless collision integral $\cC(f)$ is given by the formula
\begin{equation}\label{cc1}
\cC(f)(w)=\iint_{\bR^3\times\bS^2}(f(w')f(w'_*)-f(w)f(w_*))c(w-w_*,\om)\,\dd w_*\dd \om ,
\end{equation}
for each measurable $f$ defined a.e. on $\bR^3$ and rapidly decaying at infinity, where 
\begin{equation}\label{cc2}
\ba
w'\equiv\,w'(w,w_*,\om):=w\,-(w-w_*)\cdot\om\om\,,
\\
w'_*\equiv\!w'_*(w,w_*,\om):=w_*\!+(w-w_*)\cdot\om\om\,,
\ea
\end{equation}
(see formulas (3.11) and (4.16) in chapter II of \cite{Cerci75}). The collision kernel $c$ is of the form
\begin{equation}\label{cc3}
c(w-w_*,\om)=|w-w_*|\si_{gg}(|w-w_*|,|\cos(\widehat{w-w_*,\om})|),
\end{equation}
where $\si_{gg}$ is the dimensionless differential cross-section of gas molecules. In other words, 
$$
\Si_{gg}(|z|)=4\pi\int_0^1\si_{gg}(|z|,\mu)\,\dd\mu\,,
$$
while
\begin{equation}
\label{Pigg}
\Pi_{gg}(w,\dd W\,\dd W_*)=\iint_{\bR^3\times\bS^2}\de_{w'(w,w_*,\om)}\otimes\de_{w'_*(w,w_*,\om)}c(w-w_*,\om)\,\dd w_*\dd\om\,.
\end{equation}
where this last formula is to be understood as explained in section 2.3.1 of \cite{BDGR}. Specifically, for each test function $\phi\equiv\phi(W,W_*)\in C_b(\bR^3\times\bR^3)$, 
$$
\ba
\iint_{\bR^3\times\bR^3}\phi(W,W_*)\Pi_{gg}(w,\dd W\,\dd W_*)&
\\
=\iint_{\bR^3\times\bS^2}\phi(w'(w,w_*,\om),w'_*(w,w_*,\om))c(w-w_*,\om)\,\dd w_*\dd\om&\,.
\ea
$$

\smallskip
We also assume that the molecular interaction is defined in terms of a hard potential satisfying Grad's cutoff assumption. In other words, we assume:

\medskip
\noindent
\textbf{Assumption A1}: There exists $c_*>1$ and $\g\in[0,1]$ such that
\begin{equation}\label{cc4}
\ba
{}&0\,<\,c(z,\om)\,\le\, c_*(1+|z|)^\g\,,&&\quad\hbox{ for a.e. }(z,\om)\in\bR^3\times\bS^2\,,
\\
&\int_{\bS^2}c(z,\om)\,\dd\om\ge\frac1{c_*}\frac{|z|}{1+|z|}\,,&&\quad\hbox{ for a.e. }z\in\bR^3\,.
\ea
\end{equation}

\medskip
We next review in detail the properties of the linearization of the collision integral $\cC$ about a uniform Maxwellian $M$, which is defined by the formula
\begin{equation}\label{defL}
\cL\phi:=-M^{-1}D\cC(M)\cdot(M\phi)\,,
\end{equation}
where $D$ designates the (formal) Fr\'echet derivative. Without loss of generality, one can choose the uniform Maxwellian to be given by the expression
\begin{equation}\label{maxw}
M(w):=\tfrac1{(2\pi)^{3/2}}e^{-|w|^2/2}\,,
\end{equation}
after some Galilean transformation eliminating the mean velocity of $M$, and some appropriate choice of units so that the temperature and pressure associated to this Maxwellian state are  both equal to $1$.

We recall the following theorem, due to Hilbert (in the case of hard spheres) and Grad (in the case of hard cutoff potentials): see Theorem I on p.186 and Theorem II on p.187 in \cite{Cerci75}.

\begin{theorem}
The linearized collision integral $\cL$ is an unbounded operator on $L^2(Mdv)$ with domain 
$$
\Dom\cL=L^2((\bar c\star M)^2 Mdv)\,,\quad\hbox{ where }\bar c(z):=\int_{\bS^2}c(z,\om)d\om\,.
$$
Moreover, $\cL=\cL^*\ge 0$, with nullspace
\begin{equation}\label{null}
\Ker\cL=\Span\{1,w_1,w_2,w_3,|w|^2\}.
\end{equation}
Finally, $\cL$ is a Fredholm operator, so that
$$
\Img\cL=\Ker\cL^\bot\,.
$$
\end{theorem}

\medskip
Defining
\begin{equation}\label{defA}
A(w):=w\otimes w-\tfrac13|w|^2I
\end{equation}
the traceless part of the quadratic tensor field $w\otimes w$, elementary computations show that
$$
A\bot\Ker\cL\quad\hbox{ in }L^2(Mdv)\,.
$$
Hence there exists a unique $\tilde A\in\Dom\cL$ such that 
\begin{equation}\label{defAtilde}
\cL\tilde A=A\,,\qquad\tilde A\bot\Ker\cL\,,
\end{equation}
by the Fredholm alternative applied to the Fredholm operator $\cL$.

Using that the linearized collision integral and the tensor field $A$ are equivariant under the action of the orthogonal group, one finds that the matrix field $\tilde A$ is in fact of the form 
\begin{equation}\label{defalpha}
\tilde A(w)=\a(|w|)A(w)\,.
\end{equation}
See \cite{dego} or Appendix 2 of \cite{GoB}.

In the sequel, we shall present results which are specific to the case when both $\a$ and its derivative $\a'$ are bounded. More precisely, we make the following assumption.

\medskip
\noindent
\textbf{Assumption A2}: there exists a positive constant $C$ such that
$$
|\tilde A(w)| \le C (1 + |w|^2)\,,\quad\hbox{ and }\quad|\nabla \tilde A(w)| \le C (1 + |w|^2)\,.
$$

\medskip
We recall that, in the case of Maxwell molecules, that is, for a collision kernel $c$ of the form
$$
c(z,\om)=C(|\cos(\widehat{z,\om})|)
$$
the scalar $\a$ is a positive constant. 

\subsubsection{The collision integrals $\cD$ and $\cR$ for elastic collisions}\label{sec232}


For each measurable $F$ and $f$ defined a.e. on $\bR^3$ and rapidly decaying at infinity, the dimensionless collision integrals $\cD(F,f)$ and $\cR(f,F)$ are given by the formulas
$$
\ba
\cD(F,f)(v)&=\iint_{\bR^3\times\bS^2}(F(v'')f(w'')\!-\!F(v)f(w))b(\eps v-w,\om)\,\dd w\dd\om\,,
\\
\cR(f,F)(w)&=\iint_{\bR^3\times\bS^2}(f(w'')F(v'')\!-\!f(w)F(v))b(\eps v-w,\om)\,\dd v\dd\om\,,
\ea
$$
where
\begin{equation}\label{ela1}
\ba
{}&v''\equiv v''(v,w,\om)\,:=v\,-\frac{2\eta}{1+\eta}\!\left(v-\!\frac1\eps w\!\right)\!\cdot\om\om\,,
\\
&w''\!\equiv w''(v,w,\om)\!:=w-\frac{2}{1+\eta}\,\,(\,w-\eps v)\cdot\om\om\,,
\ea
\end{equation} 
(see formula (5.10) in chapter II of \cite{Cerci75}).
 
The collision kernel $b$ is of the form
\begin{equation}\label{ela2}
b(\eps v-w,\om)=|\eps v-w|\si_{pg}(|\eps v-w|,|\cos(\widehat{\eps v-w,\om})|),
\end{equation}
where $\si_{pg}$ is the dimensionless differential cross-section of gas molecules. In other words, 
\begin{equation}\label{ela3}
\Si_{pg}(|z|)=4\pi\int_0^1\si_{pg}(|z|,\mu)\,\dd\mu\,,
\end{equation}
while
\begin{equation}\label{ela4}
\ba
\Pi_{pg}(v,\dd V\dd W)&=\iint_{\bR^3\times\bS^2}\de_{v''(v,w,\om)}\otimes\de_{w''(v,w,\om)}b(\eps v-w,\om)\dd w\dd\om\,,
\\
\Pi_{gp}(w,\dd V\dd W)\!&=\iint_{\bR^3\times\bS^2}\de_{v''(v,w,\om)}\otimes\de_{w''(v,w,\om)}b(\eps v-w,\om)\dd v\dd\om\,,
\ea
\end{equation}
where the equalities (\ref{ela4}) are to be understood in the same way as (\ref{Pigg}).

As in the case of the molecular collision kernel $c$, we assume that $b$ is a cutoff kernel associated with a hard potential, i.e. we assume that there exists $b_*>1$ and $\b^*\in[0,1]$ such that
\begin{equation}\label{ela5}
\ba
{}&0<b(z,\om)\le b_*(1+|z|)^{\b^*}\,,&&\quad\hbox{ for a.e. }(z,\om)\in\bR^3\times\bS^2\,,
\\
&\int_{\bS^2}b(z,\om)\,\dd\om\ge\frac1{b_*}\frac{|z|}{1+|z|}\,,&&\quad\hbox{ for a.e. }z\in\bR^3\,.
\ea
\end{equation}
We also assume that, for any $p>3$,
\begin{equation}
\label{colliel}
\iiint |b(\eps v - w,\omega) - b(w, \omega)| (1+ |v|^2 + |w|^2) M(w) (1+|v|^2)^{-p} d\omega dvdw = O(\eps).
\end{equation}
This assumption is satisfied as soon as the angular cutoff in the hard potential is smooth, or in the case of hard spheres collisions.

\subsubsection{An inelastic model of collision integrals $\cD$ and $\cR$}\label{sec233}


Dust particles or droplets are macroscopic objects when compared to gas molecules. This suggests using the classical models of gas surface interactions to describe the impingement of gas molecules on dust particles or 
droplets. The simplest such model of collisions has been introduced by F. Charles in \cite{FCharlesRGD08}, with a detailed discussion in section 1.3 of \cite{FCharlesPhD} and in \cite{FCharlesSDelJSeg}. We briefly recall
this model below.

First, the (dimensional) particle-molecule cross-section is
$$
S_{pg}=\pi(r_g+r_p)^2,
$$
where $r_g$ is the molecular radius and $r_p$ the radius of dust particles or droplets. In other words, we assume that the collisions between gas molecules and the particles in the dispersed phase are hard sphere collisions.
Then, the dimensionless particle-molecule cross-section is
$$
\Si_{pg}(|\eps v-w|)=1\,.
$$
The formulas for $S_{pg}$ and $\Si_{pg}$ correspond to a binary collision between two balls of radius $r_p$ and $r_g$.

Next, the measure-valued functions $\Pi_{pg}$ and $\Pi_{gp}$ are defined by the following formulas:
\begin{equation}\label{ine1}
\ba
\Pi_{pg}(v,\dd V\,\dd W)&:=K_{pg}(v,V,W)\,\dd V\dd W\,,
\\
\Pi_{gp}(w,\dd V\,\dd W)&:=K_{gp}(w,V,W)\,\dd V\dd W\,,
\ea
\end{equation} 
where\begin{equation}\label{ine2}
\ba
K_{pg}(v,V,W):&=\tfrac1{2\pi^2}\left(\tfrac{1+\eta}\eta\right)^4\b^4\eps^3\exp\left(-\tfrac12\b^2\left(\tfrac{1+\eta}\eta\right)^2\left|\eps v-\frac{\eps V+\eta W}{1+\eta}\right|^2\right)
\\
&\qquad\times\int_{\bS^2}(n\cdot(\eps V-W))_+\left(n\cdot\left(\frac{\eps V+\eta W}{1+\eta}-\eps v\right)\right)_+dn\,,
\ea
\end{equation}
\begin{equation}\label{ine3}
\ba
K_{gp}(w,V,W)&:=\tfrac1{2\pi^2}(1+\eta)^4\b^4\exp\left(-\tfrac12\b^2(1+\eta)^2\left|w-\frac{\eps V+\eta W}{1+\eta}\right|^2\right)
\\
&\qquad\times\int_{\bS^2}(n\cdot(\eps V-W))_+\left(n\cdot\left(w-\frac{\eps V+\eta W}{1+\eta}\right)\right)_+dn\,.
\ea
\end{equation}
In the formulas above, 
$$
\beta=\sqrt{\frac{m_g}{2 k_B T_{surf}}}
$$
where $k_B$ is the Boltzmann constant and $T_{surf}$ the surface temperature of the particles.


\section{Hypothesis on $\Pi_{pg}$ and $\Pi_{gp}$}\lb{S-S3}


In the sequel, we shall provide a theorem which holds for all diffusion and friction operators satisfying a few assumptions described below. We recall that $\Pi_{pg}$ and $\Pi_{gp}$ are nonnegative measure valued functions 
of the variable $v\in\bR^3$ and of the variable $w\in\bR^3$ respectively, which depend in general on $\eps$ and $\eta$ (cf. formulas (\ref{ela1}), (\ref{ela4}), and  (\ref{ine1}) -- (\ref{ine3})).  We do not systematically mention 
this dependence, unless absolutely necessary, as in Assumptions (H4)-(H5) below. In that case, we write $\Pi_{pg}^{\eps, \eta}$ and $\Pi_{gp}^{\eps, \eta}$ instead of $\Pi_{pg}$ and $\Pi_{gp}$, as already explained.

\medskip
\noindent
\textbf{Assumption (H1)}: There exists a nonnegative function $q\equiv q(r)$, such that $q(r)\le C(1+r)$ for some constant $C>0$, and such that the measure-valued functions $\Pi_{pg}$ and $\Pi_{gp}$ satisfy
$$
\int_{\bR^3}\Pi_{pg}(v,\dd V\dd W)\,\dd v=\int_{\bR^3}\Pi_{gp}(w,\dd V\dd W)\,\dd w=q(|\eps V-W|)\,\dd V\dd W\,.
$$
Observe that Assumption (H1) is consistent with the fact that the same cross section $\Si_{pg}$ appears in the last two lines of (\ref{Colli2}). In particular, (H1) implies the local conservation law of mass.

\medskip 
\noindent
\textbf{Assumption (H2)}: There exists a nonnegative function $Q\equiv Q(r)$ in $C^1(\bR_+^*)$ such that $Q(r) + |Q'(r)|\le C(1+r)$ for some constant $C>0$, and such that the measure-valued functions $\Pi_{pg}$ and $\Pi_{gp}$ 
satisfy
$$
\ba
\eps\int_{\bR^3}(v-V)\Pi_{pg}(v,\dd V\dd W)\,\dd v&=-\eta\int_{\bR^3}(w-W)\Pi_{gp}(w,\dd V\dd W)\,\dd w
\\
&=-\frac{\eta}{1+\eta}(\eps V-W)Q(|\eps V-W|)\dd V\dd W\,.
\ea
$$
This assumption implies the conservation of momentum between molecules and particles.

\medskip
\noindent
\textbf{Assumption (H3)}: There exists a constant $C>0$ such that the measure-valued function $\Pi_{pg}$ satisfies
$$
\int_{\bR^3}\left|\eps v-\frac{\eps V+\eta W}{1+\eta}\right|^2\Pi_{pg}(v,\dd V\dd W)\,\dd v\le C\,\eta^2\,(1+|\eps V-W|^2)q(|\eps V-W|)\,\,,
$$
where $q$ is the function in Assumption (H1).

\medskip
\noindent
\textbf{Assumption (H4)}:
In the limit as $(\eps,\eta)\to(0,0)$, one has $\Pi^{\eps,\eta}_{gp}(w,\cdot)\to\Pi^{0,0}_{gp}(w,\cdot)$ weakly in the sense of probability measures for a.e. $w\in\bR^3$, and the limiting measure $\Pi^{0,0}_{gp}$ satisfies the following invariance
condition:
\be\lb{InvPi00}
\cT_R\#\Pi^{0,0}_{gp}=\Pi^{0,0}_{gp}\quad\hbox{ for each }R\in O_3(\bR)\,,
\ee
where
\begin{equation}\label{defTR}
\cT_R:\,(w,V,W)\mapsto(Rw,V,RW)\,.
\end{equation} 

Besides, for each $p>3$ and each $\Phi := \Phi(w,W)$ in $C^1(\bR^3\times\bR^3)$ such that
$$
|\Phi(w,W)|+|\nabla_w\Phi(w,W)|\le C(1+|w|^2+|W|^2)M(W)
$$
for some $C>0$, one has
$$
\ba
\int_{\bR^3}(1+ |V|^2)^{-p}\left|\iint_{\bR^3\times\bR^3}\Phi(w,W)(\Pi^{\eps,\eta}_{gp}(w,\dd V\dd W)\,\dd w-\Pi^{0,0}_{gp}(w,\dd V\dd W)\,\dd w)\right|
\\
=O(\eps+\eta)
\ea
$$
as $(\eps,\eta)\to 0$.

\medskip
\noindent
\textbf{Assumption (H5)}: There exists a positive constant $C>0$ (independent of $\eps,\eta$) such that, for all $(\eps,\eta)$ close to $(0,0)$ and for all $h \in L^2(M(w)\dd w)$, 
$$
\ba  
\iiint_{\bR^3\times\bR^3\times\bR^3}\frac{(1+|W|^2)}{(1+|V|^2)^p}M(W)|h(W)|(1+|w|^2)\Pi^{\eps,\eta}_{gp}(w,\dd V\dd W)\,\dd w&
\\
\le C\|h\|_{L^2(M(w)\dd w)}&\,.
\ea
$$

\medskip
Assumptions (H1)-(H3) and (H5) are the same as the assumptions introduced in section 3 of \cite{BDGR}. Assumption (H4) differs from its counterpart in \cite{BDGR}: while the asymptotic invariance condition (\ref{InvPi00}) is the same as in
assumption (H4) in section 3 of \cite{BDGR}, the second part of (H4) in the present paper postulates a convergence rate $O(\eps+\eta)$, whereas the second part of assumption (H4) in \cite{BDGR} only requires that the same quantity should
vanish in the limit as $(\eps,\eta)\to 0$. Besides, assumption (H4) in \cite{BDGR} involved some additional decay condition on $\Pi_{gp}^{0,0}$ which is useless here.

\bigskip
We recall that the elastic and inelastic models previously introduced (in subsections \ref{sec232} and \ref{sec233} resp.) satisfy the assumptions (H1)-(H3) and (H5): see section 3 of \cite{BDGR} for a detailed verification. It remains to verify
(H4), with the modified asymptotic condition used in the present work. These verifications are summarized in the following propositions.

We begin with the elastic collision model.

\medskip
\begin{proposition}\label{hypelast}
We consider a cross-section $b$ of the form (\ref{ela2}) satisfying (\ref{ela5})-(\ref{colliel}), and the quantities $\Sigma_{pg}$, $\Pi_{pg}$ and $\Pi_{gp}$ defined by (\ref{ela1}), (\ref{ela3}) and (\ref{ela4}). Then, assumptions (H1) -- (H5) are 
satisfied, with 
\begin{equation}\label{H11}
q(|\eps v-w|)=4\pi\int_0^1 |\eps v-w|\,\sigma_{pg}(|\eps v-w|,\mu)d\mu , 
\end{equation}
\begin{equation}\label{H12} 
Q(|\eps v-w|)=8\pi\int_0^1 |\eps v-w|\,\sigma_{pg}(|\eps v-w|,\mu)\mu^2d\mu. 
\end{equation}
\end{proposition} 

\begin{proof} 
The reader is referred to the proof of Proposition 1 in section 3 of \cite{BDGR}. We present only the part of the argument concerning the second condition in Assumption (H4), which is new. For each $p>3$ and each $\Phi:=\Phi(w,W)$ in $C^1(\bR^3\times\bR^3)$ such that
$$
|\Phi(w,W)|+|\nabla_w\Phi(w,W)|\le C(1+|w|^2+|W|^2)M(W)\,,
$$
we deduce from the mean value theorem that
$$
\ba
\int_{\bR^3}(1+ |V|^2)^{-p}\left|\iint_{\bR^3\times\bR^3}\Phi(w,W)(\Pi^{\eps,\eta}_{gp}(w,\dd V\dd W)\,\dd w-\Pi^{0,0}_{gp}(w,\dd V\dd W)\,\dd w)\right| 
\\
=\int_{\bR^3}(1+ |v|^2)^{-p}\left|\iint_{\bR^3\times\bS^2}(\Phi(w'',w)b(\eps v-w,\om)-\Phi(S_\om w)b(w,\om))\,\dd w\dd\om\right|\,\dd v
\\
\le  C\iiint_{\bR^3\times\bR^3\times\bS^2}\frac{(1+|w|^2+|v|^2)}{(1+ |v|^2)^p}M(w)|b(\eps v-w,\om)-b(w,\om)|\,\dd v\dd w\dd\om
\\
+C\iiint_{\bR^3\times\bR^3\times\bS^2}\!\!\frac{(\eta|w|\!+\!\eps|v|)}{(1\!+\!|v|^2)^p}\phi_{\eps,\eta}(w,\om,\th)b(w,\omega)\,\dd v\dd w\dd\om
\\
\le C(\eta+\eps)+C'(\eta+\eps)\iint_{\bR^3\times\bR^3}\frac{(1\!+\!|w|)(|w|\!+\!|v|)(1\!+\!|w|^2\!+\!|v|^2)}{(1\!+\!|v|^2)^p}M(w)\,\dd v\dd w 
\\
\le C''(\eta+\eps)\,. 
\ea
$$
where
$$
S_\om w:=w-2 (w\cdot\om)\om\,,
$$
and
$$
\ba
\phi_{\eps,\eta}(w,\om,\th):=\sup_{0<\th<1}|\nabla_w\Phi(S_\om w\!+\!\tfrac{2\theta}{1+\eta}(\eta w\!+\!\eps v)\!\cdot\!\om\om,w)|
\\
\le C(1+2|w|^2+4|\eta w+\eps v|^2+|w|^2)M(w)
\\
\le C(1+11|w|^2+8|v|^2)M(w)
\ea
$$
for all $0<\eps,\eta<1$.
\end{proof}

\medskip
Next we check assumptions (H1)-(H5) on the inelastic collision model.

\begin{proposition}\label{hypinelast}
Consider the measure-valued functions $\Pi_{pg}$ and $\Pi_{gp}$ defined in (\ref{ine1})-(\ref{ine3}). Then, assumptions (H1)-(H5) are satisfied, with 
$$ 
q(|\eps v-w|)=|\eps v-w|
$$
and
$$ 
Q(|\eps v-w|)=\frac{\sqrt{2\pi}}{3\b}+|\eps v-w|\,.
$$
\end{proposition}

\begin{proof}
Here again, we present only the argument justifying the second part of (H4) which is new, and refer to Proposition 2 in section 3 of \cite{BDGR} for a complete proof of the remaining statements.

With the substitution $w \mapsto z= (1+\eta) w - \eps V - \eta W$, 
$$ 
\ba 
\iint_{\bR^3\times\bR^3}\Phi(w,W)(1+\eta)^4\exp\left(-\frac12\beta^2(1+\eta)^2\left|w-\frac{\eps V+\eta W}{1+\eta}\right|^2\right)
\\
\times\int_{\bS^2}((\eps V - W)\cdot n)_+\left(\left(\frac{\eps V+\eta W}{1+\eta}-w\right)\cdot n\right)_+\,\dd n\dd w\dd W 
\\
=\iint_{\bR^3\times\bR^3}\Phi\left(\tfrac{z+\eps V+\eta W}{1+\eta},W\right)e^{-\tfrac12\beta^2|z|^2}J_\eps(V,W)\dd z\dd W\,,
\ea
$$
where
$$
J_\eps(V,W,z):=\int_{\bS^2}((\eps V-W)\cdot n)_+(n \cdot z)_+\,\dd n\,.
$$
We shall also denote
$$
J(W,z):=J_\eps(0,W,z)\,,\qquad\hat J(W,z)=J(W,z)+J(-W,z)\,.
$$

Then, for each $p>3$ and each $\Phi:=\Phi(w,W)$ in $C^1(\bR^3\times\bR^3)$ satisfying
$$
|\Phi(w,W)|+|\nabla_w\Phi(w,W)|\le C(1+|w|^2+|W|^2)M(W)\,,
$$
one has
$$ 
\ba
\int_{\bR^3}(1+ |V|^2)^{-p} \left|\iint_{\bR^3\times\bR^3}\Phi(w,W)(\Pi^{\eps,\eta}_{gp}(w,\dd V\dd W)\,\dd w-\Pi^{0,0}_{gp}(w,\dd V\dd W)\,\dd w)\right|&
\\
\le C\iiint\frac{e^{-\frac12\beta^2|z|^2}}{(1+|V|^2)^p}|\Phi(\tfrac{z+\eps V+\eta W}{1+\eta},W)J_\eps(V,W,z)-\Phi(z,W)J(W,z)|\,\dd z\dd V\dd W&
\\
\le C\eps\iiint\frac{e^{-\frac12\beta^2|z|^2}}{(1+|V|^2)^p}|\Phi(\tfrac{z+\eps V+\eta W}{1+\eta},W)|\hat J(V,z)\,\dd z\dd V\dd W& 
\\
+C\iiint\frac{e^{-\frac12\beta^2|z|^2}}{(1+|V|^2)^p}|\Phi(\tfrac{z+\eps V+\eta W}{1+\eta},W)-\Phi(z,W)|J(W,z)\,\dd z\dd V\dd W&\,.
\ea
$$
By the mean value theorem
$$
\ba
|\Phi(\tfrac{z+\eps V+\eta W}{1+\eta},W)-\Phi(z,W)|\le|\eps V+\eta(W-z)|\sup_{0<\th<1}|\grad_w\Phi(z+\th\tfrac{\eps V+\eta(W-z)}{1+\eta})|&
\\
\le|\eps V+\eta(W-z)|(1+(|z|+|\eps V|+\eta|W-z|)^2+|W|^2)M(W)&
\\
\le(\eta|z|+\eta|W|+\eps|V|)(1+6|z|^2+3|V|^2+7|W|^2)M(W)&\,,
\ea
$$
while
$$
\ba
|\Phi(\tfrac{z+\eps V+\eta W}{1+\eta},W)|\le C(1+|z+\eps V+\eta W|^2+|W|^2)M(W)&
\\
\le C(1+3|z|^2+3|V|^2+4|W|^2)M(W)&\,,
\ea
$$
if $0<\eps,\eta<1$. Thus
$$ 
\ba
\int_{\bR^3}(1+ |V|^2)^{-p} \left|\iint_{\bR^3\times\bR^3}\Phi(w,W)(\Pi^{\eps,\eta}_{gp}(w,\dd V\dd W)\,\dd w-\Pi^{0,0}_{gp}(w,\dd V\dd W)\,\dd w)\right|&
\\
\le 4C\eps\iiint\frac{e^{-\frac12\beta^2|z|^2}}{(1+|V|^2)^p}[[z,V,W]]M(W)\hat J(V,z)\,\dd z\dd V\dd W&
\\
+7C\iiint\frac{e^{-\frac12\beta^2|z|^2}}{(1+|V|^2)^p}(\eta|z|+\eta|W|+\eps|V|)[[z,V,W]]M(W)J(W,z)\,\dd z\dd V\dd W&
\\
\le 4C\eps\iiint\frac{e^{-\frac12\beta^2|z|^2}}{(1+|V|^2)^p}[[z,V,W]]M(W)|V||z|\,\dd z\dd V\dd W&
\\
+7\sqrt{3}C\max(\eps,\eta)\iiint\frac{e^{-\frac12\beta^2|z|^2}}{(1+|V|^2)^p}[[z,V,W]]^{3/2}M(W)J(W,z)\,\dd z\dd V\dd W&
\\
\le C'(\eps+\eta)&\,.
\ea
$$
We have denoted
$$
[[z,V,W]]:=1+|z|^2+|V|^2+|W|^2
$$
and used the Cauchy-Schwarz inequality
$$
\eta|z|+\eta|W|+\eps|V|\le\sqrt{3}\max(\eps,\eta)(|z|^2+|W|^2+|V|^2)^{1/2}\,.
$$
\end{proof}


\section{Passing to the limit}\lb{S-S4}


\subsection{Statement of the main result}


We now consider a sequence of solutions $f_n\equiv f_n(t,x,w)\ge 0$, and $F_n\equiv F_n(t,x,v)\ge 0$ to the system of kinetic-fluid equations (\ref{BoltzSysSc2}), with $\eps,\eta,\mu$ replaced with sequences $\eps_n,\eta_n,\mu_n\to 0$ 
respectively. Thus
\begin{equation}\label{kifu}
\ba
{}&\d_tF_n+v\cdot\grad_xF_n=\frac1{\eta_n}\cD(F_n,f_n),
\\
&\d_tf_n+\frac1{\eps_n}w\cdot\grad_xf_n=\frac{1}{\mu_n}\cR(f_n,F_n)+\frac{\mu_n}{\eps_n^2}\cC(f_n),
\ea
\end{equation}
where $\cC,\cD$ and $\cR$ are defined by (\ref{cc1}), (\ref{cc3}), (\ref{newd}) and (\ref{newr}).

Our main result is stated below.

\begin{theorem} \label{theor}
Let $g_n\equiv g_n(t,x,w)$ and $F_n\equiv F_n(t,x,v)\ge 0$ be sequences of smooth (at least $C^1$) functions. Assume that $F_n$ and $f_n$ defined by 
\begin{equation}\label{kifu2}
f_n(t,x,w)=M(w)(1+\eps_n g_n(t,x,w)),
\end{equation} 
where $M$ is given by (\ref{maxw}), are solutions to the system (\ref{kifu}), where $\cC$, $\cD$ and $\cR$ are defined by (\ref{cc1})-(\ref{cc3}) and (\ref{newd}) and (\ref{newr}). We assume moreover that $\Pi_{pg}^{\eps_n,\eta_n}$ and 
$\Pi_{gp}^{\eps_n,\eta_n}$ in (\ref{newd})-(\ref{newr}) satisfy assumptions (H1)-(H5) (with $\Si_{pg}$ defined by (\ref{Colli2}) in accordance with Assumption (H1)). We also assume that the molecular interaction verifies assumptions (A1) 
and (A2).

\smallskip
Assume that 
$$
\eps_n\to 0\,,\quad\eta_n/\eps_n^2\to 0\,,\quad\eps_n/\mu_n^2\to 0\,,\quad\mu_n\to 0\,,
$$
that 
$$
F_n\wto F\quad\hbox{ in }L^\infty_{loc}(\bR_+^*\times\bR^3\times\bR^3)\hbox{ weak-*}\,,
$$ 
and that
$$
g_n\wto g\hbox{ in }L^2_{loc}(\bR_+^* \times \bR^3 \times \bR^3)\hbox{ weak.}
$$
(a) Assume that $F_n$ decays sufficiently fast, uniformly in $n$, in the velocity variable; in other words assume that, for some $p>3$,
$$
\sup_{n\ge 1}\sup_{t,|x|\le R}\sup_{v\in\bR^3}(1+|v|^2)^p\,F_n(t,x,v)<\infty
$$
for all $R>0$. 

\noindent
(b) Assume that, for some $q>1$,
$$
\sup_{t,|x|<R}\int_{\bR^3}(1+|w|^2)^q\,M(w)\,g_n^2(t,x,w)\,\dd w<\infty
$$
for all $R>0$.

\smallskip
Then there exist $L^{\infty}$ functions $\rho\equiv\rho(t,x)\in\bR$, $\th\equiv\th(t,x)\in\bR$ and a velocity field $u\equiv u(t,x)\in\bR^3$ s.t. for a.e $t,x \in \bR_+^* \times \bR^3$,
\begin{equation} \label{eqr}
g(t,x,w)=\rho(t,x)+u(t,x)\cdot w+\th(t,x)\tfrac12(|w|^2-3)\,,
\end{equation}
while $u,F$ satisfies the Vlasov-Stokes system
\begin{equation} \label{VNS}
\left\{
\ba
{}&\d_tF+v\cdot\grad_xF=\ka\Div_v((v-u)F),
\\	\\
&\Div_xu=0,
\\	\\
&-\nu\Dlt_xu+\grad_xp=\ka\int_{\bR^3}(v-u)F\,\dd v,
\ea
\right.
\end{equation}
in the sense of distributions, with 
\begin{equation} \label{nuka}
\nu:=\tfrac1{10}\int\tilde A:\cL\tilde A M\,\dd w>0\,,\quad\ka:=\tfrac13\int Q(|w|)|w|^2M\,\dd w>0,
\end{equation}
where $Q$ is defined in assumption (H2), and $\tilde A$, $\cL$ are defined by (\ref{defAtilde}), (\ref{defL}).
\end{theorem}


\subsection{Proof of Theorem \ref{theor}}


We split this proof in several steps, summarized in Propositions \ref{rhout} to \ref{last}, and a final part in which the convergence of all the terms in eq. (\ref{kifu}) is established.


\subsubsection{Step 1: Asymptotic form of the molecular distribution function}


We first identify the asymptotic structure of the fluctuations of molecular distribution function about the Maxwellian state $M$.

\medskip
\begin{proposition}\label{rhout}
Under the same assumptions as in  Theorem \ref{theor}, there exist two functions $\rho\equiv\rho(t,x)\in\bR$, $\th\equiv\th(t,x)\in\bR$ and a vector field $u\equiv u(t,x)\in\bR^3$ satisfying
$$
\rho,\th\in L^\infty(\bR_+\times\bR^3)\,,\qquad u\in L^\infty(\bR_+\times\bR^3;\bR^3)
$$
such that the limiting fluctuation $g$ of molecular distribution function about $M$ is of the form (\ref{eqr}) for a.e $t,x \in \bR_+^* \times \bR^3$. 

Moreover, $u$ satisfies the divergence-free condition
$$
\Div_xu=0\,.
$$

Finally
$$
\int_{\bR^3}\tilde A(w)(w\cdot\grad_xg)M\,\dd w=\nu(\grad_xu+(\grad_xu)^T)
$$
where $\tilde A$ is defined in (\ref{defAtilde}), and $\nu$ is defined in (\ref{nuka}).
\end{proposition}

\begin{proof}
Since $\cC$ is a quadratic operator, its Taylor expansion terminates at order $2$, so that
$$
\ba
\cC(M(1+\eps_ng_n))=&\cC(M)+\eps_nD\cC(M)\cdot(Mg_n)+\eps_n^2\cC(Mg_n)
\\
=&-\eps_nM\cL g_n+\eps_n^2M\cQ(g_n)\,,
\ea
$$
where $\cL$ is defined by (\ref{defL}) and
\begin{equation} \label{defQ}
\cQ(\phi):=M^{-1}\cC(M\phi)\,.
\end{equation}
Then the kinetic equation for the propellant (second line of eq. (\ref{kifu})) can be recast in terms of the fluctuation of the distribution function $g_n$ as follows:
\begin{equation} \label{eqg}
\d_tg_n+\frac1{\eps_n}w\cdot\grad_xg_n+\frac{\mu_n}{\eps_n^2}\cL g_n=\frac1{\mu_n\eps_n}M^{-1}\cR(M(1+\eps_n g_n),F_n)+\frac{\mu_n}{\eps_n}\cQ(g_n)\,.
\end{equation}

Multiplying each side of this equation by $\eps_n^2/\mu_n$ shows that
\begin{equation} \label{eqgpr}
\cL g_n=\frac{\eps_n}{\mu_n^2}M^{-1}\cR(M(1+\eps_n g_n),F_n)+\eps_n\cQ(g_n)-\frac{\eps_n^2}{\mu_n}\d_tg_n-\frac{\eps_n}{\mu_n}w\cdot\grad_xg_n\,.
\end{equation}

The last two terms of this identity clearly converge to $0$ in the sense of distributions, since $g_n\wto g$ in $L^2_{loc}(\bR_+^*\times\bR^3\times\bR^3)$ weak.
 
Next, we observe that, for $w',w'_*$ defined by (\ref{cc2}) and $\phi \in C_c(\bR^3)$, one has
$$ 
\ba
\int_{\bR^3}\cQ(g_n)(w)\phi(w)\,\dd w&
\\
=\iiint\left(\frac{\phi(w')}{M(w')}-\frac{\phi(w')}{M(w')}\right)M(w_*)g_n(w_*)M(w)g_n(w)c(w-w_*,\om)\,\dd\om\dd w_*\dd w&\,.
\ea
$$
By the Cauchy-Schwarz inequality,
$$
\ba
\left|\int_{\bR^3}\cQ(g_n)(w)\phi(w)\,\dd w\right|
\\
\le C\iint_{\bR^3\times\bR^3}M(w_*)g_n(w_*)M(w)g_n(w)(1+|w|+|w_*|)\,\dd w_*\dd w  
\\
\le C\int_{\bR^3}M(w)g_n(w)^2\,\dd w\,\int_{\bR^3}M(w)(1+|w|)^2\,\dd w\,.
\ea
$$
Therefore
$$
\int_{\bR^3}\cQ(g_n)(w)\phi(w)\,\dd w\hbox{ is bounded in }L^1_{loc}(\bR_+^*\times\bR^3)
$$
for each $\phi \in C_c(\bR^3)$, so that 
$$
\eps_n \cQ(g_n)\to\hbox{ in }\cD'(\bR_+^*\times\bR^3)\,.
$$

Similarly
$$ 
\ba
\int_{\bR^3}\cR(f_n, F_n)M^{-1}(w)\phi(w)\,\dd w
\\
=\iiint_{\bR^3\times\bR^3\times \bR^3}\left(\frac{\phi(w)}{M(w)}-\frac{\phi(v)}{M(v)}\right)f_n(W)F_n(V)\Pi_{gp}(w,\dd V\dd W)\,\dd w 
\ea
$$
so that 
$$ 
\ba
\left|\int_{\bR^3}\cR(f_n, F_n)M^{-1}(w)\phi(w)\,\dd w\right|\le C\iint_{\bR^3\times\bR^3}F_n(V)f_n(W)q(|\eps_n V - W|)\,\dd V\dd W& 
\\
\le C\int_{ \bR^3}M(W)(1+|g_n|)(W)(1+|W|)\,\dd W&\,,
\ea
$$
because of assumptions (a)-(b). Therefore, the sequence
$$
\int_{\bR^3}\cR(f_n, F_n)M^{-1}(w)\phi(w)\,\dd w
$$
is bounded in $L^\infty_{loc}(\bR_+^*\times\bR^3)$, and hence
$$
\frac{\eps_n}{\mu_n^2}M^{-1}\cR(f_n, F_n)\to 0\hbox{ in }\cD'(\bR_+^*\times\bR^3\times\bR^3)\,.
$$ 

Hence, we deduce from (\ref{eqgpr}) that
$$
\cL g_n\to 0\hbox{ in }\cD'(\bR_+^*\times\bR^3\times\bR^3)\,.
$$
On the other hand, assumption (b) and the fact that $g_n\wto g$ in $L^2_{loc}(\bR_+^*\times\bR^3\times\bR^3)$ imply that
$$
\cL g_n\wto\cL g\quad\hbox{ in }L^2_{loc}(\bR_+^*\times\bR^3)\hbox{ weak.}
$$
Therefore $\cL g=0$. 

Since $\Ker\cL$ is the linear span of $\{1,v_1,v_2,v_3,|v|^2\}$, this implies the existence of $\rho,\th\in L^2_{loc}(\bR_+\times\bR^3)$ and of $u\in L^2_{loc}(\bR_+^*\times\bR^3;\bR^3)$ such that (\ref{eqr}) holds.
That $\rho,\th\in L^\infty_{loc}(\bR_+\times\bR^3)$ and $u\in L^\infty_{loc}(\bR_+^*\times\bR^3;\bR^3)$ follows from assumption (b) and the formulas
$$
\rho=\int_{\bR^3}g M\,\dd w\,,\quad u=\int_{\bR^3}wg M\,\dd w\,,\quad\th=\int_{\bR^3}(\tfrac13|v|^2-1)g M\,\dd w\,.
$$

The remaining statements, i.e. the divergence free condition satisfied by $u$ and the computation of
$$
\int_{\bR^3}\tilde A(w)(w\cdot\grad_xg)M\,\dd w
$$
are obtained as in \cite{BDGR} --- specifically, as in Propositions 6 and 7 of \cite{BDGR} respectively.

\end{proof}

\smallskip
\begin{remark}
In the case of elastic collisions between the gas molecules and the particles in the dispersed phase, using more carefully the symmetries in the collision integrals leads to an estimate of the form
$$ 
\left|\int_{\bR^3}\cR(f_n, F_n)M^{-1}(w)\phi(w)\,\dd w\right|\leq C(\eps_n+\eta_n)\,,
$$
so that the assumption $\eps_n/\mu_n\to 0$ (instead of $\eps_n/\mu_n^2\to 0$) is enough to guarantee that
$$
\frac{\eps_n}{\mu_n^2}M^{-1}\cR(f_n, F_n)\to 0\hbox{ in }\cD'(\bR_+^*\times\bR^3\times\bR^3)\,.
$$
The same is true for the inelastic collision model if $\beta=1$, i.e. if the surface temperature of the particles or droplets is equal to the temperature of the Maxwellian around which the distribution function of the gas is linearized.
\end{remark}


\subsubsection{Step 2: Asymptotic deflection and friction terms}


The following result can then be proved exactly as in \cite{BDGR} (more precisely, see Propositions 4 and 5 in \cite{BDGR}).

\begin{proposition}\label{defl}
Under the same assumptions as in Theorem \ref{theor},
$$
\ba
{}&\frac1{\eta_n}\cD(F_n,f_n)\to\ka\Div((v-u)F)&&\quad\hbox{ in }\cD'(\bR_+^*\times\bR^3\times\bR^3),
\\
&\frac1{\eps_n}\int w\cR(f_n,F_n)dw\to\ka\int(v-u)Fdv&&\quad\hbox{ in }\cD'(\bR_+^*\times\bR^3),
\ea
$$
with $\ka$ defined by eq. (\ref{nuka}). 
\end{proposition}


\subsubsection{Step 3: Asymptotic friction flux}


The asymptotic friction flux is handled as in Proposition 9 in \cite{BDGR}, with some modifications due to the differences in the scalings used here and in \cite{BDGR}.

\begin{proposition}\label{last}
Under the same assumptions of Theorem \ref{theor},
$$
\frac1{\mu_n}\int_{\bR^3}\tilde A(w)\cR(f_n,F_n)(w)\,\dd w\to 0\quad\hbox{ in }\cD'(\bR_+^*\times\bR^3)\,.
$$
\end{proposition}

\begin{proof}
First, we compute
$$ 
\ba
\int_{\bR^3}\tilde A(w)\cR(M,F_n)(w)\,\dd w
\\
=\iiint_{\bR^3\times\bR^3\times\bR^3}F_n(V)M(W)(\tilde A(w)-\tilde A(W))\Pi^{\eps_n,\eta_n}_{gp}(w,\dd V\dd W)\,\dd w\,. 
\ea
$$
We see that
$$
\ba
\left|\int\tilde A(w)\cR(M,F_n)(w)\,\dd w\!-\!\iiint F_n(V)M(W)(\tilde A(w)\!-\!\tilde A(W))\Pi^{0,0}_{gp}(w,\dd V\dd W)\,\dd w\right|&
\\
\le C\int(1+|V|^2)^{-p}\left|\iint \Phi(w,W)(\Pi^{\eps_n,\eta_n}_{gp}(w,dVdW) - \Pi^{0,0}_{gp}(w,dVdW)) dw\right|&\,,
\ea
$$
with
$$
\Phi(w,W)=M(W)(\tilde A(w)-\tilde A(W))\,.
$$
By assumption (A2), $\Phi\in C^1(\bR^3\times\bR^3)$ and 
$$
|\Phi(w,W)|+|\nabla_w\Phi(w,W)|\le C(1+|w|^2+|W|^2)M(W)\,.
$$ 
Therefore, the second part of Assumption (H4) implies that
$$ 
\ba
\int_{\bR^3}\tilde A(w)\cR(M,F_n)(w)\,\dd w&
\\
= \iiint F_n(V)M(W)(\tilde A(w)-\tilde A(W))\Pi^{0,0}_{gp}(w,\dd V\dd W)\,\dd w+O(\eps_n+\eta_n)&\,.
\ea
$$
We conclude by using the symmetry assumption (first part of Assumption (H4)) as in \cite{BDGR}, and arrive at the bound
\be\lb{FricFlux1}
\sup_{t+|x|<R}\int_{\bR^3}\tilde A(w)\cR(M,F_n)(w)\,\dd w=O(\eps_n+\eta_n)
\ee
for all $R>0$.

Finally, observe that, for some $p>3$, one has
$$ 
\ba
\left|\int_{\bR^3}\cR(Mg_n, F_n)\tilde A(w)\,\dd w\right|&
\\
\le C_p\iiint_{\bR^3\times\bR^3\times\bR^3}\frac{|w|^2+|W|^2}{(1+|V|^2)^p}M(W)|g_n(W)|\Pi_{gp}(w,\dd V\dd W)\,\dd w&\,, 
\ea
$$
where $C_p\equiv C_p(t,x)\in L^\infty_{loc}(\bR_+^*\times\bR^3)$. The integral on the right hand side of this last inequality is bounded in $L^\infty_{loc}(\bR_+^*\times\bR^3)$ by (H5) and assumption (b) in Theorem \ref{theor}.

Since $\frac{\eps_n}{\mu_n}\to 0$,
$$
\frac{\eps_n}{\mu_n}\int_{\bR^3}\tilde A(w)\cR(Mg_n,F_n)\,\dd w\to 0\quad\hbox{ in }\cD'(\bR_+^*\times\bR^3)\,.
$$
With (\ref{FricFlux1}), this concludes the proof since $f_n= M(1+\eps_ng_n)$.
\end{proof}


\subsubsection{Step 4: End of the proof of Theorem \ref{theor}}


For simplicity, we henceforth use the notation
$$
\la\phi\ra:=\int_{\bR^3}\phi(w)M(w)\,\dd w\,.
$$

Since $\cL = \cL^*$
$$
\frac{\mu_n}{\eps_n}\la A(w)g_n\ra=\frac{\mu_n}{\eps_n}\la(\cL\tilde A)(w)g_n\ra=\La\tilde A(w)\frac{\mu_n}{\eps_n}\cL g_n\Ra\,.
$$
Then, we use the Boltzmann equation for $g_n$ written in the form (\ref{eqgpr}) to express the term $\frac1{\eps_n}\cL g_n$:
$$
\ba
\frac{\mu_n}{\eps_n}\la A(w)g_n\ra=&\mu_n\la\tilde A(w)\cQ(g_n)\ra-\la\tilde A(w)(\eps_n\d_t+w\cdot\grad_x)g_n\ra
\\
&+\frac1{\mu_n}\la\tilde A(w)M^{-1}\cR(f_n,F_n)\ra\,.
\ea
$$

We first pass to the limit in $\la\tilde A(w)(\eps_n\d_t+w\cdot\grad_x)g_n\ra$ in the sense of distributions, using assumption (A2) and assumption (b) in Theorem \ref{theor}:
$$ 
\la\tilde A(w)(\eps_n\d_t+w\cdot\grad_x)g_n\ra\to\la\tilde A(w)w\cdot\grad_xg\ra\,,\quad\hbox{ in }\cD'(\bR_+^*\times\bR^3)\,.
$$

Let
$$ 
P(w,w_*):=\int_{\bS^2}(\tilde A(w') - \tilde A(w))c(w-w_*,\om)\,\dd\om\,.
$$
Then
$$ 
\la\tilde A\cQ(g_n)\ra=\iint_{\bR^3\times\bR^3}P(w,w_*)M(w_*)g_n(w_*)M(w)g_n(w)\,\dd w_*\dd w\,.
$$
Clearly $|P(w,w_*)|\le C(1+|w|^3+|w_*|^3)$ by assumption (A2). Then, 
$$
\ba
|\la\tilde A(w)\cQ(g_n)\ra|\le\left(\iint_{\bR^3\times\bR^3}\!\!\! M(w_*)g^2_n(w_*) M(w)g^2_n(w)\,\dd w_*\dd w\right)^{1/2}&
\\
\times\left(\iint_{\bR^3\times\bR^3}\!\!\! M(w_*) M(w)(1+|w|^3+|w_*|^3)^2\,\dd w_*\dd w\right)^{1/2}&\,,
\ea
$$
so that
$$ 
\mu_n\la\tilde A(w)\cQ(g_n)\ra\to 0\hbox{ in }\cD'(\bR_+^*\times\bR^3)\,.
$$
By Proposition \ref{last}, 
$$
\frac1{\eps_n}\la A(w)g_n\ra\to -\nu\left((\grad_xu)+(\grad_xu)^T\right)\quad\hbox{ in }\cD'(\bR_+\times\bR^3),
$$
so that 
$$
\Div_x\frac{\mu_n}{\eps_n}\la A(w)g_n\ra\to -\nu\Dlt_xu-\nu\grad_x\Div_xu=-\nu\Dlt_xu
$$
in $\cD'(\bR_+^*\times\bR^3)$ since $u$ is divergence free by Proposition \ref{rhout}.

Hence, for each divergence free test vector field $\xi\equiv\xi(x)\in\bR^3$,
$$
\ba
\int_{\bR^3}\frac{\mu_n}{\eps_n}\la w\otimes wg_n\ra:\grad\xi\,\dd x=\int_{\bR^3}\frac{\mu_n}{\eps_n}\la A(w)g_n\ra:\grad\xi\,\dd x
\to-\nu\int_{\bR^3}\grad_xu:\grad\xi\,\dd x
\ea
$$ 
in $\cD'(\bR_+^*)$.

Multiplying both sides of (\ref{eqg}) by $wM(w)$ and integrating over $\bR^3$, we see that
\begin{equation}\label{newlin}
\d_t\la wg_n\ra+\frac1{\eps_n}\Div_x\la w\otimes wg_n\ra=\frac1{\mu_n\eps_n}\la wM^{-1}\cR(f_n,F_n)\ra\,.
\end{equation}

By Proposition \ref{rhout}, 
$$
\la wg_n\ra\to\la wg\ra=u
$$
in  $\cD'(\bR_+^*\times\bR^3)$, while, by Proposition \ref{defl},
$$
\frac1{\eps_n}\la wM^{-1}\cR(f_n,F_n)\ra\to\ka\int_{\bR^3}(v-u)F\,\dd v
$$
in  $\cD'(\bR_+^*\times\bR^3)$. Thus, for each divergence free test vector field $\xi\equiv\xi(x)\in\bR^3$, passing to the limit in the weak formulation (in $x$) of the momentum balance law (\ref{newlin}), i.e.
$$
\mu_n\d_t\int\xi\cdot\la wg_n\ra-\frac{\mu_n}{\eps_n}\int_{\bR^3}\la A(w)g_n\ra:\grad\xi\,\dd x=\frac1{\eps_n}\int_{\bR^3}\xi\cdot\la wM^{-1}\cR(f_n,F_n)\ra\,\dd x\,,
$$
results in
\be\lb{WFormStokes}
0=-\nu\int_{\bR^3}\grad_xu:\grad\xi\,\dd x+\ka\iint_{\bR^3\times\bR^3}\xi\cdot(v-u)F\,\dd v\dd x\,.
\ee
In other words, let $T=(T_1,T_2,T_3)\in\cD'(\bR^3;\bR^3)$ be defined as follows:
$$
T:=\nu\Dlt_xu+\ka\iint_{\bR^3}\xi\cdot(v-u)F\,\dd v\,.
$$
Then (\ref{WFormStokes}) is equivalent to the fact that
$$
\sum_{k=1}^3\ll T_k,\xi_k\gg_{\cD',C^\infty_c}=0
$$
for each test vector field $\xi\in C^\infty_c(\bR^3)$ such that $\Div\xi=0$. (In the identity above, we have denoted $\ll ,\gg_{\cD',C^\infty_c}$ the pairing between distributions and compactly supported $C^\infty$ functions.)

By de Rham's characterization of currents homologous to $0$ (see Thm. 17' in \cite{deRham}), this implies the existence of $p\in\cD'(\bR^3)$ such that
$$
T=\grad_xp\,.
$$
Thus, the fact that the identity (\ref{WFormStokes}) holds for each test vector field $\xi\in C^\infty_c(\bR^3)$ such that $\Div\xi=0$ is the weak formulation (in the sense of distributions) of the last equation in (\ref{VNS}).

Finally, the equation for the distribution function $F$ of the dispersed phase (i.e. the first line of (\ref{kifu})) is
$$
\d_tF_n+v\cdot\grad_xF_n=\frac1{\eta_n}\cD(F_n,f_n).
$$
Since $F_n\to F$ in $L^\infty_{loc}(\bR_+^*\times\bR^3\times\bR^3)$, one has
$$
\d_tF_n+v\cdot\grad_xF_n\to\d_tF+v\cdot\grad_xF\quad\hbox{ in }\cD'(\bR_+^*\times\bR^3\times\bR^3)\,.
$$
By Proposition \ref{defl}, 
$$
\d_tF+v\cdot\grad_xF=\ka\Div_v((v-u)F)\,,
$$
which is the first equation in (\ref{VNS}).

This concludes the proof of Theorem \ref{theor}.


\bigskip
\noindent
{\bf{Acknowledgment}}: The research leading to this paper was funded by the French ``ANR blanche'' project Kibord: ANR-13-BS01-0004, and by Universit\'e Sorbonne Paris Cit\'e, in the framework of the ``Investissements d'Avenir'', 
convention ANR-11-IDEX-0005. V.Ricci acknowledges the support by the GNFM (research project 2015: ``Studio asintotico rigoroso di sistemi a una o pi\`u componenti'').


\end{document}